\documentclass[a4paper,leqno]{article}
\usepackage[headings]{fullpage}

\usepackage{amssymb,amsmath,amstext,amsthm,mathtools}
\usepackage{enumerate,url,comment}

\usepackage{cmbright}

\usepackage[colorlinks=true,urlcolor=cyan,linkcolor=blue,citecolor=magenta]{hyperref}

\usepackage{xpatch}
\xpatchcmd{\qed}{\hfill}{}{}{}
\xpatchcmd{\qed}{\quad}{\qquad}{}{}

\newtheorem{theorem}{Theorem}
\newtheorem{proposition}[theorem]{Proposition}
\newtheorem{lemma}[theorem]{Lemma}
\newtheorem{corollary}[theorem]{Corollary}
\newtheorem{conjecture}[theorem]{Conjecture}
\newtheorem{question}[theorem]{Question}

\theoremstyle{definition}
\newtheorem{definition}[theorem]{Definition}
\newtheorem{remark}[theorem]{Remark}

\newcommand{\lk}{\mathrm{lk}}
\newcommand{\anst}{\mathrm{ast}}
\newcommand{\aff}{\mathrm{Aff}}
\newcommand{\conv}{\mathrm{conv}}
\newcommand{\vertices}{\mathrm{Vert}}
\newcommand{\st}{\mathrm{st}}

\newcommand{\lex}{\mathrm{Lex}}

\newcommand{\set}[2]{\left\{ #1\,\middle|\, #2\right\}}

\newcommand{\gc}{g^c}
\newcommand{\hc}{h^c}
\newcommand{\hsc}{h^{sc}}
\newcommand{\gsc}{g^{sc}}
\newcommand{\dhalf}[1]{\left\lfloor #1/2\right\rfloor}

\newcommand{\R}{\mathbb{R}}

\newcommand{\mw}{{\sf MW}}

\def\mchoose#1#2{\ensuremath{\left(\kern-.3em\left(\genfrac{}{}{0pt}{}{#1}{#2}\right)\kern-.3em\right)}}

\newcommand{\vfig}[2]{{\raisebox{.1em}{$#1$}\left/\raisebox{-.1em}{$#2$}\right.}}

\title{On the cone of $f$-vectors of cubical polytopes}

\author{Ron M. Adin\\
{\it Department of Mathematics, Bar-Ilan University}\\
\url{radin@math.biu.ac.il}
\and Daniel Kalmanovich\\
{\it Einstein Institute of Mathematics, The Hebrew University of Jerusalem}\\
\url{daniel.kalmanovich@gmail.com}
\and Eran Nevo\\
{\it Einstein Institute of Mathematics, The Hebrew University of Jerusalem}\\
\url{nevo.eran@gmail.com}
}

\date{April 23, 2018}

\usepackage{fancyhdr}
\begin{document}

\maketitle

\begin{abstract}
	What is the minimal closed cone containing all $f$-vectors of cubical $d$-polytopes?
	We construct cubical polytopes showing that this cone, expressed in the cubical $g$-vector coordinates, contains the nonnegative $g$-orthant,
	thus verifying one direction of the Cubical Generalized Lower Bound Conjecture of Babson, Billera and Chan. Our polytopes also show that a natural cubical analogue of the simplicial Generalized Lower Bound Theorem does not hold.
\end{abstract}

\section{Introduction}
Understanding the possible face numbers of polytopes, and of subfamilies of interest, is a fundamental question, dealt with since antiquity.
The celebrated $g$-theorem, conjectured by McMullen~\cite{McMu71} and proved by Stanley~\cite{Stan80} (necessity) and by Billera and Lee~\cite{BillL81} (sufficiency), characterizes the $f$-vectors of simplicial polytopes.
Here we consider $f$-vectors of cubical polytopes;
a $d$-polytope $Q$ is \emph{cubical} if all its facets are combinatorially isomorphic to the $(d-1)$-cube.
Adin~\cite{Adin96} proved analogues of the Dehn--Sommerville relations for cubical polytopes,
thus encoding the $f$-vector of $Q$ by its (long) cubical $g$-vector
\[g^c(Q)=\left(g_1^c(Q), g_2^c(Q), \dots, g_{\dhalf{d}}^c(Q)\right)\]
(with the constant $g_0^c(Q)=2^{d-1}$ omitted).
The construction of neighborly cubical $d$-polytopes by Joswig and Ziegler~\cite{JoswZ00}, where the number of vertices varies, shows that the linear span of the vectors $g^c(Q)$, over all cubical $d$-polytopes, is the entire vector space $\mathbb{R}^{\lfloor d/2\rfloor}$.
Adin~\cite[Question 2]{Adin96} asked whether $g^c(Q)$ is always in the nonnegative orthant, and Babson, Billera and Chan~\cite[Conjecture 5.2]{BabsBC97} conjectured further that the minimal closed cone $\mathcal{C}_d$ containing all the vectors $g^c(Q)$ corresponding to cubical $d$-polytopes is exactly this nonnegative orthant $\mathcal{A}_d$.

Denote by $e_i$ the $i$-th unit vector in $\mathbb{R}^{\lfloor d/2\rfloor}$. Stacked cubical polytopes show that the ray spanned by $e_1$ is in $\mathcal{C}_d$, and neighborly cubical polytopes show that the ray spanned by $e_{\lfloor d/2\rfloor}$ is in $\mathcal{C}_d$.
Our main result is that all the rays spanned by the vectors $e_i$ are in $\mathcal{C}_d$. Thus

\begin{theorem}\label{thm:cones}
$\mathcal{A}_d \subseteq \mathcal{C}_d$.
\end{theorem}

The conjecture of Adin and Babson--Billera--Chan is
\begin{conjecture}\label{conj:cones}
$\mathcal{A}_d = \mathcal{C}_d$.
\end{conjecture}

An analogue of Theorem~\ref{thm:cones} was previously known for the much wider class of PL cubical spheres~\cite[Theorem 5.7]{BabsBC97}.
Also, Conjecture~\ref{conj:cones} holds for $d\leq 5$, by combining the constructions above with Steve Klee's result \cite[Prop.3.7]{Klee11} asserting that $g^c_k(Q)\geq 0$ for any cubical polytope $Q$ of dimension $2k+1$.

Sanyal and Ziegler~\cite{SanyZ10} showed how to construct, from any simplicial $(d-2)$-polytope $P$ on $n-1$ vertices and a total order $v_1<v_2<\ldots <v_{n-1}$ on its vertices, a cubical $d$-polytope $Q=Q(P,<)$ on $2^n$ vertices; it is the projection of a deformed $n$-cube in $\mathbb{R}^n$ onto the last $d$ coordinates.
Further, they showed that if $P$ is $k$-neighborly then
the $k$-skeleton of $Q$ is isomorphic to the $k$-skeleton of the $n$-cube.
We apply their construction to the case where $P_n$ is the $k$-neighborly $k$-stacked $(d-2)$-polytope on $n-1$ vertices constructed by McMullen and Walkup~\cite{McMuW71}, with $1\leq k\leq \lfloor\frac{d-2}{2}\rfloor$, and with a suitable total order $<$. Analyzing the cubical $g$-vectors of the resulting polytopes $Q(k,d,n)=Q(P_n,<)$, as $n$ tends to infinity, gives~\autoref{thm:cones}. See Theorem~\ref{t:gc_Q} and Corollary~\ref{t:asymptotics} for the exact values and asymptotic behavior of the cubical $g$-vectors.

The generalized lower bound theorem for simplicial polytopes (GLBT), conjectured by McMullen and Walkup~\cite{McMuW71} and proved by Murai and Nevo~\cite{MuraN13}, asserts that for $1\le k<\lfloor\frac{d}{2}\rfloor$, a simplicial $d$-polytope $P$ is $k$-stacked if and only if $g_{k+1}(P)=0$.
The polytopes $Q(k,d,n)$ demonstrate that the natural cubical analogue of the GLBT does not hold:
\begin{theorem}\label{thm:IntroNoCGLBT}
	For any $k\ge 1$ and $n\ge d\ge 2k+4$, we have $\gc_{k+2}(Q(k,d,n))=0$, and $Q(k,d,n)$ is not cubical $(k+1)$-stacked.
\end{theorem}


The paper is organized as follows. Basic definitions and notation are given in~\autoref{sec:preliminaries}. In~\autoref{sec:MW} we define our variant of the McMullen--Walkup polytopes. A sketch of the Sanyal--Ziegler construction is given in~\autoref{sec:SZ}. In~\autoref{sec:Q} we construct the polytopes $Q(k,d,n)$ mentioned above and analyze their cubical $g$-vector. In~\autoref{sec:stack} we prove Theorem~\ref{thm:IntroNoCGLBT}, showing that $Q(k,d,n)$ is not $(k+1)$-stacked. The final~\autoref{sec:concluding_remarks} concludes with remarks and open questions.

\subsection*{Acknowledgments}

The authors thank Michael Joswig and Isabella Novik for helpful comments.
The research of Adin was supported by an MIT-Israel MISTI grant. He also thanks the Israel Institute for Advanced Studies, Jerusalem, for its hospitality during part of the work on this paper.
The research of Kalmanovich and Nevo was partially supported by Israel Science Foundation grant ISF-1695/15 and by grant 2528/16 of the ISF-NRF Singapore joint research program.
This work was also partially supported by the National Science Foundation under Grant No.\ DMS-1440140, while Nevo was in residence at the Mathematical Sciences Research Institute in Berkeley, California, during the Fall 2017 semester.

An extended abstract version of this paper will appear in the FPSAC 2018 proceedings. We thank the FPSAC 2018 referees for their comments on that extended abstract.

\section{Preliminaries}\label{sec:preliminaries}

The purpose of this section is mainly to set the notation that we will use throughout the paper. For undefined terminology we refer the reader to~\cite{Zieg95}. A {\bf $d$-dimensional polytope} $P$ is the convex hull of a finite set of points in $\R^d$ which affinely span $\R^d$. A (proper) {\bf face} $\sigma$ of $P$ is the intersection of $P$ with one of its supporting hyperplanes, the {\bf dimension} $\dim\sigma$ of $\sigma$ is then the dimension of the affine span of that intersection. The faces of dimensions $0$, $1$, and $d-1$ are called {\bf vertices}, {\bf edges}, and {\bf facets}, respectively. The empty set $\emptyset$ and the polytope $P$ itself are called {\bf trivial faces} and have dimensions $-1$ and $d$, respectively. We will abbreviate and write $d$-polytope and $i$-face to denote dimension.
We denote by $\vertices(P)$ the set of vertices of $P$, and for a vertex $v\in\vertices(P)$, we denote by $\vfig{P}{v}$ the {\bf vertex figure} of $P$ at $v$, that is, $\vfig{P}{v}$ is a $(d-1)$-polytope obtained as the intersection of $P$ with a hyperplane which strictly separates $v$ from $\vertices(P)\setminus\{v\}$; the face lattice of $\vfig{P}{v}$ does not depend on the seperating hyperplane chosen.

A {\bf polytopal complex} $K$ is a finite collection of polytopes in $\R^d$ such that
\begin{enumerate}[(i)]
\item the empty polytope is in $K$,
\item if $P\in K$ then all faces of $P$ are also in $K$,
\item if $P,Q\in K$ then $P\cap Q$ is a face of both $P$ and $Q$.
\end{enumerate}
The {\bf dimension} $\dim K$ of $K$ is the maximum of $\dim P$ over all $P\in K$; we say that $K$ is a $\dim K$-complex. 
The elements in $K$ are called {\bf faces}; the faces of dimension $\dim K$ are called {\bf facets}.
For $F\in K$ we define the (open) {\bf star} of $F$ and the {\bf antistar} of $F$, respectively, by
\begin{align*}
\st_F(K)&=\set{G\in K}{F\subseteq G},\\
\anst_F(K)&=\set{G\in K}{F\nsubseteq G}.
\end{align*}
The number of $i$-faces in $K$ is denoted by $f_i(K)$, and the {\bf $f$-vector} of $K$ is $f(K)=\left(f_0(K),f_1(K),\dots ,f_{\dim K}(K)\right)$. The {\bf $f$-polynomial} of $K$ is defined by
\[
f(K,t)=\sum_{i=0}^{\dim K+1} f_{i-1}(K)t^i,
\]
where $f_{-1}(K) = 1$ for any nonempty $K$.

For a polytope $P$ we denote by $\langle P\rangle$ the complex of all faces of $P$. The {\bf boundary complex} $\partial P$ is the complex formed by all the proper faces of $P$, that is $\partial P=\langle P\rangle\setminus \{P\}$. We also define the $f$-vector and $f$-polynomial of $P$ by $f(P) = f(\partial P)$ and $f(P,t) = f(\partial P,t)$.
We use $\lk_v(P)$ to denote the boundary complex $\partial \left(\vfig{P}{v}\right)$ of the vertex figure of $P$ at $v$.

\subsection{Simplicial complexes and polytopes}

A {\bf simplicial complex} is a polytopal complex in which all polytopes are simplices. Let $K$ be a simplicial $(D-1)$-complex; the {\bf $h$-polynomial} of $K$ is defined by
\begin{equation*}
\begin{aligned}
h(K,t)
&=h_0(K)+h_1(K)t+\dots +h_D(K)t^D \\
&:=(1-t)^D\cdot f\left(K,\frac{t}{1-t}\right),
\end{aligned}
\end{equation*}
and the {\bf $h$-vector} of $K$ is $(h_0(K),\dots ,h_D(K))$. If $K=\partial P$ for a simplicial $D$-polytope $P$ then the {\bf Dehn--Sommerville relations} assert that $h_i(K)=h_{D-i}(K)$ for any $0\le i\le D$.
The corresponding {\bf $g$-vector} $\left(g_0(K),\dots ,g_{\dhalf{D}}(K)\right)$ of $K$ is then defined by
\begin{equation*}
\begin{aligned}
g_0(K)&= h_0(K) = 1,\\
g_i(K)&=h_i(K)-h_{i-1}(K),\quad\text{for all}\quad 1\leq i\leq\dhalf{D}.
\end{aligned}
\end{equation*}

For two simplicial complexes $K$ and $L$ we define the {\bf join} $K\ast L$ to be the simplicial complex whose simplices are the disjoint unions of simplices of $K$ and simplices of $L$.

A polytope is {\bf simplicial} if each of its proper faces is a simplex.
For a simplicial polytope $P$ we write $h(P,t)$ to mean $h(\partial P,t)$, and similarly $h_i(P):=h_i(\partial P)$ and $g_i(P):=g_i(\partial P)$.

A simplicial $d$-polytope $P$ is called {\bf $k$-stacked} if $P$ has a triangulation in which there are no interior faces of dimension less than $d-k$. A simplicial polytope $P$ is called {\bf $k$-neighborly} if each subset of at most $k$ vertices forms the vertex set of a face of $P$. We denote by $C(d,n)$ the {\bf cyclic $d$-polytope with $n$ vertices}:
$$C(d,n) := \conv\left\{x(t_1), x(t_2). . . , x(t_n)\right\},$$
where $t_1 < t_2 <\dots < t_n$ and $x(t) := \left(t,t^2,\dots ,t^d\right)$ is the moment curve in $\R^d$.

\subsection{Cubical complexes and polytopes}

A {\bf cubical complex} is a polytopal complex in which all polytopes are combinatorially isomorphic to cubes. Let $Q$ be a cubical $(d-1)$-complex, the {\bf short cubical $h$-polynomial} is defined by
\[\hsc(Q,t)=
\sum_{i=0}^{d-1} \hsc_i(Q) t^i
= \sum_{j=0}^{d-1} f_j(Q) (2t)^j (1-t)^{d-1-j}.
\]
When $Q$ is clear from the context, we may sometimes drop it from the notation, as we do in the following few definitions.
The {\bf (long) cubical $h$-vector} $(\hc_0,\hc_1,\dots ,\hc_d)$ is defined by
\begin{align*}
\hc_0&=2^{d-1},\\
\hsc_i&=\hc_i+\hc_{i+1},\quad\text{for}\quad 0\leq i\leq d-1,
\end{align*}
and the corresponding ({\bf short} and {\bf long}) {\bf cubical $g$-vector}s are defined, as in the simplicial case, by
\begin{align*}
\gsc_0=\hsc_0=f_0,& &  \gsc_i&=\hsc_i -\hsc_{i-1}\quad\text{for } 1\leq i\leq\dhalf{(d-1)}; \\
\gc_0=\hc_0=2^{d-1},& &  \gc_i&=\hc_i -\hc_{i-1}\quad\text{for } 1\leq i\leq\dhalf{d}.
\end{align*}

A polytope is {\bf cubical} if each of its proper faces is combinatorially a cube.
Adin~\cite{Adin96} showed that any cubical $d$-polytope $Q$ satisfies an analogue of the Dehn-Sommerville relations: $\hc_i(Q)=\hc_{d-i}(Q)$ for all $0\le i\le d$.

In analogy with the simplicial case, \cite{BabsBC97} defined cubical neighborliness and cubical stackedness: a cubical $d$-polytope is {\bf $k$-neighborly} if it has the $k$-skeleton of a cube (of some dimension);
it is {\bf $k$-stacked} if it has a cubical subdivision with no interior faces of dimension less than $d-k$.

Each vertex figure in a cubical $d$-polytope $P$ is a simplicial $(d-1)$-polytope;  we finish this section with the relation known as {\bf Hetyei's observation}:
\begin{equation}\label{eq:Hetyei}
\hsc(P,t):=\hsc(\partial P,t)=\sum_{v\in \vertices(P)}h(\lk_v(P),t).
\end{equation}
It shows that the cubical Dehn--Sommerville relations follow from the simplicial ones.

\section{The McMullen--Walkup polytopes}\label{sec:MW}

In section 3 of~\cite{McMuW71}, McMullen and Walkup describe the construction of $k$-neighborly $k$-stacked simplicial $D$-polytopes with $N$ vertices, for any set of parameters satisfying $2\leq 2k\leq D < N$. Their construction takes a $k$-neighborly $2k$-polytope $C$ with $N-D+2k$ vertices (e.g., the cyclic $2k$-polytope with $N-D+2k$ vertices) and a $(D-2k)$-simplex $T$, both lying in $\R^D$ in such a way that the relative interior of $T$ intersects the affine hull $\aff (C)$ in a vertex $x$ of $C$. Then the convex hull $\conv (C\cup T)$ is the desired polytope. We define a slightly more general construction.
\begin{definition}\label{def:MW}
Let $2 \leq K \leq D < N$. Let $C=C(K,N-D+K)$ be the cyclic $K$-polytope with $N-D+K$ vertices, and let $T$ be a $(D-K)$-simplex, both lying in $\R^D$ in such a way that the relative interior of $T$ intersects $\aff (C)$ in a vertex $x$ of $C$. The polytope $\conv(C\cup T)$ is a $D$-dimensional simplicial polytope with $N$ vertices, denoted $\mw(K,D,N;x)$.
\end{definition}

The faces of $\mw(K,D,N;x)$ are of two types:
\begin{enumerate}[(a)]
\item the convex hull of $T$ and a face of $C$ that contains $x$,
\item the convex hull of a proper face of $T$ and a face of $C$ that does not contain $x$.
\end{enumerate}
The boundary complex of $\mw(K,D,N;x)$ is thus described by
\begin{equation}\label{eq:P}
\partial\mw(K,D,N;x)=\langle T\rangle \ast \lk_x(C)\bigcup_{\partial T\ast \lk_x(C)}\partial T\ast \anst_x(C).
\end{equation}

McMullen and Walkup have shown that $\mw(2k,D,N;x)$ is $k$-neighborly and $k$-stacked, thus satisfying
\begin{equation}\label{eq:g_of_P}
g_i(\mw(2k,D,N;x))=\begin{cases}
0, & i> k;\\
\mchoose{N-D-1}{i}=\binom{N-D+i-2}{i}, & i\leq k.
\end{cases}
\end{equation}

The proof~\cite[p.~269]{McMuW71} that $\mw(2k,D,N;x)$ is $k$-neighborly and $k$-stacked actually shows:

\begin{lemma}\label{lem:gMW}
The polytope $\mw(K,D,N;x)$ is $\dhalf{K}$-neighborly and $\dhalf{K}$-stacked. In particular, for $K = 2k-1$,
\begin{equation}
g_i(\mw(2k-1,D,N;x))=\begin{cases}
0, & i> k-1;\\
\mchoose{N-D-1}{i}=\binom{N-D+i-2}{i}, & i\leq k-1.
\end{cases}
\end{equation}
\end{lemma}

The vertices of $C$ come with a natural total order $v_1<v_2<\dots <v_{N-D+K}$, inherited from the order $t_1 < t_2 <\dots < t_{N-D+K}$ of the parameters in the definition of $C$. We take $x$ to be the last vertex in that ordering,
denoting the resulting polytope simply by $\mw(K,D,N)$.
Removing $x=v_{N-D+K}$, we extend the order $v_1<\dots<v_{N-D+K-1}$ of the remaining vertices of $C$  to an order $v_1<\dots<v_{N-D+K-1}<v_{N-D+K}<\dots<v_N$ of the vertices of $\mw(K,D,N)$, where $v_{N-D+K},\dots,v_N$ are the vertices of the $(D-K)$-simplex $T$.
We will use the following result:

\begin{lemma}\label{lem:vlink_of_MW}
The vertex figure $\vfig{\mw(2k,D,N)}{v_1}$ is combinatorially isomorphic to $\mw(2k-1,D-1,N-1)$.
\end{lemma}
\begin{proof}
For $C=C(2k,N-D+2k)$ denote $C'=\vfig{C}{v_1}$, and note that $C'\cong C(2k-1,N-D+2k-1)$. Applying the construction in Definition~\ref{def:MW} to $C'$ and a $(D-2k)$-simplex $T$ produces an $\mw(2k-1,D-1,N-1)$ with boundary complex
\begin{equation}\label{eq:boundery_complex_C'}
\langle T\rangle \ast \lk_x(C')\bigcup_{\partial T\ast \lk_x(C')}\partial T\ast \anst_x(C').
\end{equation}
We now show that the complex above is equal to $\lk_{v_1}(\mw(2k,D,N))$.
Let $F$ be a face of $\mw(2k,D,N)$ containing $v_1$.
If $F$ contains $T$ then $F\setminus\{v_1\}$ is in $\langle T\rangle \ast \lk_x(C')$,
and if $F$ does not contain $T$ then $F\setminus\{v_1\}$ is in $\partial T\ast \anst_x(C')$. Thus $\lk_{v_1}(\mw(2k,D,N))$ is contained in the complex~\eqref{eq:boundery_complex_C'}.
Similarly, for the other direction, take a face $F'$ in the complex~\eqref{eq:boundery_complex_C'} and observe that $F'\cup\{v_1\}$ is in the boundary complex of $\mw(2k,D,N)$, as described in~\eqref{eq:P}.
\end{proof}

\section{The Sanyal--Ziegler construction}\label{sec:SZ}

We give a very brief sketch of the construction, focusing on the combinatorial description of vertex figures. The reader is prompted to confer with the paper~\cite{SanyZ10}, or with Sanyal's diploma thesis~\cite{Sany05}.

Let $(P,<)$ be a simplicial $(d-2)$-polytope with $n-1$ vertices, totally ordered by $<$. Label the vertices $v_1,\dots,v_{n-1}\in\R^{d-2}$ according to the given order $v_1<v_2<\dots <v_{n-1}$, and assume that the vertices are in general position, i.e., no $d-1$ vertices of $P$ lie on a hyperplane. We start by defining the lexicographic diamonds of $P$.

\subsection{Lexicographic diamonds}

Let $w_1,\dots ,w_{n-1}\in\R$ be a set of {\bf heights}, and denote by $V^w=\set{(w_i,v_i)}{1\leq i\leq n-1}\subset\R^{d-1}$ the set of {\bf lifted vertices}. Let $p=(w_0,v'_0)\in\R^{d-1}$ be an arbitrary point with $w_0\gg |w_i|$ for every $1\leq i\leq n-1$, and consider the $(d-1)$-polytope $D(P,w)=\conv(V^w,p)$. We call $D(P,w)$ the {\bf diamond} over $P$ with
heights $w$, noting that, for $w_0$ big enough, the combinatorial type of $D(P,w)$ is independent of the choice of the point $p$.
Projecting the lower envelope of $D(P,w)$ onto $P$ shows that $\anst_p(D(P,w))$ is a polytopal subdivision of $P$.

Of special interest are the subdivisions of $P$ induced by the heights $(w_1,w_2,\dots,w_{n-1})=(\pm h,0,\dots,0)$ with $h>0$.
The subdivision of $P$ induced by $(-h,0,\dots,0)$ is obtained by {\bf pulling} $v_1$; it is a triangulation of $P$, and its facets are the pyramids with apex $v_1$ over facets in $P\cap P_1$ where $P_1=\conv(v_2,\dots,v_{n-1})$.
The subdivision of $P$ induced by $(h,0,\dots,0)$ is obtained by {\bf pushing} $v_1$, and its facets are the pyramids with apex $v_1$ over facets in $P_1\setminus P$, plus one (possibly non-simplex) facet $P_1$.
The {\bf $a$-th lexicographic subdivision} $\lex_a(P)$ of $P$ is the polytopal subdivision of $P$ obtained by successively pushing the vertices $v_1,\dots,v_{a-1}$, and then pulling $v_a$. That is, pushing $v_1$ creates a subdivision with one non-simplex cell $P_1$, which we replace by its subdivision obtained by pushing $v_2$, which, in turn, has only one non-simplex cell $P_2=\conv\{v_3,\dots ,v_{n-1} \}$, and so on, until we finally replace $P_{a-1}=\conv\{v_a,\dots, v_{n-1} \}$ by its triangulation obtained by pulling $v_a$.

The above iterative procedure amounts to choosing a set of heights $w_1,\dots w_{n-1}$ with
\[
w_1>\dots>w_{a-1}>0>w_a,\quad\text{and}\quad w_{a+1}=\dots =w_{n-1}=0.
\]
The resulting diamond (with $w_0\gg w_1,-w_a$), denoted $D_a=D(P,w)$, is called the {\bf $a$-th lexicographic diamond}.
Its vertices are labeled $v_0, v_1, \dots, v_{n-1}$,
with $v_0$ corresponding to the apex $p$; thus $\anst_{v_0}(D_a)=\lex_a(P)$.

\begin{remark}
Note that pushing or pulling a vertex in a simplex has no effect, thus the (possibly) different diamonds
are $D_a$ with $1\leq a\leq n-d+1$.
\end{remark}

\subsection{The vertex figures of $Q$}

Take a Gale transform $G\in \R^{(n-1)\times(n-d)}$ of $P$ that has the form $G=\begin{bmatrix}
I_{n-d}\\
\overline{G}
\end{bmatrix}$, where $\overline{G}\in\R^{(d-1)\times (n-d)}$. Plugging $\overline{G}$ into the {\bf deformed cube template} (see~\cite[Definition 3.1]{SanyZ10}) produces a combinatorial $n$-cube $\overline{C}=C_n(\overline{G})$. The projection $\pi_d(\overline{C})$ of $\overline{C}$ onto the last $d$ coordinates is the cubical polytope $Q=Q(P,<)$ mentioned in the introduction.

The following key result from~\cite{SanyZ10}%
\footnote{Theorem 3.7 in~\cite{SanyZ10} actually contains a typo, having $n-d-1$ instead of the correct value $n-d+1$. Their proof, however, does give the correct value.
Further, \cite[Theorem 3.7]{SanyZ10} is stated for $Q(P,<)$ where $P$ is neighborly; however, their proof holds verbatim for any simplicial polytope $P$.
}
states that each vertex figure of $Q$ is combinatorially equivalent to some diamond $D_a$ and, moreover, specifies which diamond corresponds to a given vertex $v$ of $Q$.
\begin{lemma}[{\cite[Theorem 3.7]{SanyZ10}}]\label{lem:SZ}
Let $v$ be a vertex of $\overline{C}$ labeled by $\sigma\in\{+,-\}^n$. Then the vertex figure of $\pi_d(v)$ in $Q$ is isomorphic to $D_a$ with
$$a=\min\left(\set{i\in [n]}{\sigma_i=+}\bigcup \{n-d+1\}\right).$$
\end{lemma}

The isomorphism $D_a\cong \vfig{Q}{v}$ is given by:
$v_i\in D_a$ corresponds to the neighbor of $v$ obtained by flipping the $(i+1)$-th coordinate of $v$ $(0 \leq i \leq n-1)$.

\section{The cubical polytopes $Q(k,d,n)$}\label{sec:Q}

Fix positive integers $k \ge 1$ and $n \ge d \ge 2k+2$.
We apply the Sanyal--Ziegler construction to the McMullen--Walkup polytope $P=\mw(2k,d-2,n-1)$, with a total order $<$ on its vertices as described after Lemma~\ref{lem:gMW} above.
The result is a $d$-dimensional cubical polytope $Q=Q(k,d,n)=Q(P,<)$ with $2^n$ vertices.
We now compute its cubical $g$-vector $\gc(Q)$, in stages.

\subsection{Computing $\gsc(Q(k,d,n))$}

By Hetyei's observation~\eqref{eq:Hetyei} we have
\begin{equation*}
\hsc_i(Q)=\sum_{v\in \vertices(Q)}h_i(\lk_v(Q))
\qquad (0 \leq i \leq d-1).
\end{equation*}
Therefore, by Lemma~\ref{lem:SZ}, for each $1\leq i\leq \lfloor\frac{d-1}{2}\rfloor$:
\begin{equation}\label{eq:Q_g-vector}
\gsc_i(Q)
= \sum_{v\in \vertices(Q)} g_i(\lk_v(Q))
= \sum_{a=1}^{n-d} 2^{n-a} g_i(D_a) + 2^{d} g_i(D_{n-d+1}).
\end{equation}

We now compute the $g$-vectors of the diamonds $D_a$ at hand, namely for our choice of $(P,<)$. First we compute $g(D_1)$, a computation that is then used to compute $g(D_a)$ for general $a$.

\begin{proposition}\label{prop:D_1_g-vector}
For $a=1$ and $0 \leq i\leq \lfloor\frac{d-1}{2}\rfloor$:
$$g_i(D_1)=\begin{cases}
0,& \text{if } i> k+1;\\
\mchoose{n-d}{k},& \text{if } i=k+1;\\
\mchoose{n-d}{i},& \text{if } i\leq k.
\end{cases}$$
\end{proposition}
\begin{proof}
Let $B_1=\lex_1(P)$, namely the simplicial $(d-2)$-ball triangulating $P$ by starring from $v_1$. Then
$$f(D_1,t)=f(B_1,t)+t\cdot f(P,t).$$
Note that, strictly speaking, the simplicial complexes actually considered here are $\partial D_1$, $B_1$ and $\partial P$, of dimensions $d-2$, $d-2$ and $d-3$, respectively.
The corresponding $h$-polynomial is
\[
\begin{aligned}
h(D_1,t)
&= (1-t)^{d-1} \cdot f\left(D_1,\frac{t}{1-t}\right) \\
&= (1-t)^{d-1} \cdot \left[ f\left(B_1,\frac{t}{1-t}\right) + \frac{t}{1-t} \cdot f\left(P,\frac{t}{1-t}\right)\right] \\
&= h(B_1,t)+t\cdot h(P,t).
\end{aligned}
\]
To compute $h(B_1,t)$, observe that
\[
f(B_1,t)=(1+t)\cdot f(\anst_{v_1}(P))=(1+t)\cdot \Big( f(P,t)-t\cdot f(\lk_{v_1}(P),t)\Big).
\]
The $h$-polynomial is therefore
\[
\begin{aligned}
h(B_1,t)
&= (1-t)^{d-1} \cdot f\left(B_1,\frac{t}{1-t}\right) \\
&= (1-t)^{d-1} \cdot \left(1+\frac{t}{1-t}\right)\cdot \left[ f\left(P,\frac{t}{1-t}\right)-\frac{t}{1-t}\cdot f\left(\lk_{v_1}(P),\frac{t}{1-t}\right)\right]\\
&= (1-t)^{d-2}\cdot f\left(P,\frac{t}{1-t}\right)-t\cdot (1-t)^{d-3}\cdot f\left(\lk_{v_1}(P),\frac{t}{1-t}\right)\\
&= h(P,t)-t\cdot h(\lk_{v_1}(P),t).
\end{aligned}
\]
Summing up, we have
\begin{equation*}
h(D_1,t)=(1+t)\cdot h(P,t)-t\cdot h(\lk_{v_1}(P),t)
\end{equation*}
which yields, for $1\leq i\leq \lfloor\frac{d-1}{2}\rfloor$,
\begin{equation}\label{eq:g_of_D_1}
g_i(D_1)=g_i(P)+g_{i-1}(P)-g_{i-1}(\lk_{v_1}(P)).
\end{equation}
Using~\eqref{eq:g_of_P}, Lemma~\ref{lem:vlink_of_MW} and Lemma~\ref{lem:gMW}, we compute the right-hand side of~\eqref{eq:g_of_D_1} to obtain the result.
\end{proof}

\begin{proposition}\label{prop:D_a_g-vector}
For each $2\leq a\leq n-d+1$ and $0 \leq i\leq \lfloor\frac{d-1}{2}\rfloor$:
\[
g_i(D_a)=\begin{cases}
0,& \text{if } i>k+1;\\
\mchoose{n-d-a+1}{k},& \text{if } i=k+1;\\
\mchoose{n-d}{i},& \text{if } i\leq k.
\end{cases}
\]
\end{proposition}
\begin{proof}
Let $v_0$ be the apex of the diamond $D_a$.
The link condition $\lk_{v_1v_0}(D_a)=\lk_{v_1}(D_a)\cap\lk_{v_0}(D_a)$ holds:
indeed, $\lk_{v_0}(D_a)=\partial P$ and the faces of $\lk_{v_1}(D_a)$ not containing $v_0$ form $\lk_{v_1}(\lex_a (P))$, so
\[
\lk_{v_1}(D_a)\cap\lk_{v_0}(D_a)=
\lk_{v_1}(\lex_a (P))\cap \partial P,
\]
and as $a\ge 2$, the right hand side equals
$\lk_{v_1}(\partial P)$, for which we have $\lk_{v_1}(\partial P)=
\lk_{v_1}(\lk_{v_0}(D_a))=
\lk_{v_1v_0}(D_a)$.

Denote by $K$ the $(d-2)$-complex obtained from $\partial D_a$ by contracting the edge $v_1v_0$.
Then, the $h$-polynomials are related by
\[
h(D_a,t)=h(K,t)+t\cdot h(\vfig{P}{v_1},t),
\]
see e.g.~\cite[Eq.~(10)]{nevo_VKobs}.

Note that $K$ is the boundary complex of a polytope, namely of the $(a-1)$-lexicographic diamond over $P_1=\mw(2k,d-2,n-2)$; denote this diamond by $Y_{a-1}(n-1)$.
Then, by Lemma~\ref{lem:vlink_of_MW}
\begin{equation*}
h(D_a,t)=h(Y_{a-1}(n-1),t)+t\cdot h(\mw(2k-1,d-3,n-2),t).
\end{equation*}
We may again contract the edge between the first vertex of $P_1$ and the apex $v_0$ in $Y_{a-1}(n-1)$ to obtain $Y_{a-2}(n-2)$. We do these edge contractions $a-1$ times, and get
\begin{equation*}
h(D_a,t)=h(Y_1(n-a+1),t)+t\cdot \sum_{j=1}^{a-1}h(\mw(2k-1,d-3,n-1-j),t),
\end{equation*}
and therefore
\begin{equation}
g_i(D_a)=g_i(Y_1(n-a+1))+\sum_{j=1}^{a-1}g_{i-1}(\mw(2k-1,d-3,n-1-j)).
\end{equation}
As $Y_1(n-a+1)$ is the $1$-lexicographic diamond over $\mw(2k,d-2,n-a)$,
the claimed result follows from Proposition~\ref{prop:D_1_g-vector}, Lemma~\ref{lem:gMW}, and the identity
\begin{multline*}
\mchoose{n-d-a+1}{i}+\sum\limits_{j=1}^{a-1}\mchoose{n-d-j+1}{i-1} \\
=\mchoose{n-d-a+1}{i}+\sum\limits_{j=1}^{a-1}\left[\mchoose{n-d-j+1}{i}-\mchoose{n-d-j}{i}\right] = \mchoose{n-d}{i}.\qedhere
\end{multline*}
\end{proof}

Combining~\eqref{eq:Q_g-vector} with Propositions~\ref{prop:D_1_g-vector} and~\ref{prop:D_a_g-vector},
and noting that $\mchoose{0}{k} = 0$ for $k \ge 1$,
we conclude
\begin{corollary}\label{t:gsc_Q}
For each $0 \leq i \leq \lfloor \frac{d-1}{2} \rfloor$,
\[
\gsc_i(Q)=\begin{cases}
0,& \text{if } i>k+1;\\
{\displaystyle \sum_{a=1}^{n-d}} 2^{n-a} \mchoose{n-d-a+1}{k}
,& \text{if } i=k+1;\\
2^n \mchoose{n-d}{i},& \text{if } i\leq k.
\end{cases}
\]
\end{corollary}

\subsection{Computing $\gc(Q(k,d,n))$}

In order to compute the cubical $g$-vector of $Q$, and in particular $\gc_{k+2}(Q)$, we need the following binomial identity:
\begin{lemma}\label{t:binomial}
For any integers $k \ge 1$ and $m \ge 0$,
\[
\sum_{a=1}^{m} 2^{m-a} \mchoose{m-a+1}{k}
= (-1)^{k+1} + 2^m \sum_{j=0}^{k} (-1)^j \mchoose{m}{k-j}.
\]
\end{lemma}
\begin{proof}
Fix $k \ge 1$, and denote
\[
\begin{aligned}
L_m &:= \sum_{a=1}^{m} 2^{m-a} \mchoose{m-a+1}{k}
= \sum_{b=0}^{m-1} 2^{b} \mchoose{b+1}{k}, \\
R_m &:= (-1)^{k+1} + 2^m \sum_{j=0}^{k} (-1)^j \mchoose{m}{k-j}.
\end{aligned}
\]
We will prove that $L_m = R_m$ for all $m \ge 0$, by induction on $m$.

For $m=0$, indeed $L_0 = R_0 = 0$.
For the induction step, note that
$L_{m+1} = L_m + 2^m \mchoose{m+1}{k}$.
Thus it suffices to prove that
$R_{m+1} = R_m + 2^m \mchoose{m+1}{k}$ as well. Indeed,
\begin{equation*}
\begin{aligned}
R_{m+1} - R_m
&= 2^m \sum_{j=0}^{k} (-1)^j \left[ 2\mchoose{m+1}{k-j} - \mchoose{m}{k-j} \right] \\
&= 2^m \sum_{j=0}^{k} (-1)^j \left[ \mchoose{m+1}{k-j} + \mchoose{m+1}{k-j-1} \right] \\
&= 2^m \mchoose{m+1}{k},
\end{aligned}
\end{equation*}
as claimed.
\end{proof}

\begin{theorem}\label{t:gc_Q}
For each $1 \leq i \leq \dhalf{d}$,
\[
\gc_i(Q)=\begin{cases}
0,& \text{if } i>k+1;\\
2^n {\displaystyle\sum_{j=1}^{i}} (-1)^{j-1} \mchoose{n-d}{i-j} + (-1)^i 2^d,& \text{if } i\leq k+1.
\end{cases}
\]
\end{theorem}
\begin{proof}
From
\[
\begin{aligned}
\gc_0(Q) &= 2^{d-1}, \qquad
\gsc_0(Q) = 2\gc_0(Q) + \gc_1(Q), \\
\gsc_i(Q) &= \gc_i(Q) + \gc_{i+1}(Q)
\qquad (1 \leq i \leq \dhalf{(d-1)})
\end{aligned}
\]
it follows that
\begin{equation}\label{eq:gc_Q}
\gc_i(Q) = \sum_{j=1}^{i} (-1)^{j-1} \gsc_{i-j}(Q) + (-1)^i 2^d
\qquad (1 \leq i \leq \dhalf{d}).
\end{equation}
The values of $\gc_i(Q)$ for $i \le k+1$ now follow easily from Corollary~\ref{t:gsc_Q}.
It also follows that
\[
\gc_i(Q) + \gc_{i+1}(Q) = 0 \qquad (i \ge k+2),
\]
and all that remains is to show that $\gc_{k+2}(Q) = 0$.
Indeed, by \eqref{eq:gc_Q} and Corollary~\ref{t:gsc_Q},
\[
\begin{aligned}
\gc_{k+2}(Q)
&= \sum_{j=1}^{k+2} (-1)^{j-1} \gsc_{k+2-j}(Q) + (-1)^{k+2} 2^d \\
&= \sum_{a=1}^{n-d} 2^{n-a} \mchoose{n-d-a+1}{k}
+ 2^n \sum_{j=2}^{k+2} (-1)^{j-1} \mchoose{n-d}{k+2-j} + (-1)^{k} 2^d \\
&= \sum_{a=1}^{n-d} 2^{n-a} \mchoose{n-d-a+1}{k}
- \left[ 2^n \sum_{j=0}^{k} (-1)^{j} \mchoose{n-d}{k-j} + (-1)^{k+1} 2^d \right].
\end{aligned}
\]
Using Lemma~\ref{t:binomial} with $m = n-d$ gives, indeed,
$\gc_{k+2}(Q) = 0$ as claimed.
\end{proof}

\begin{corollary}\label{t:asymptotics}
Fix $k \ge 1$ and $d \ge 2k+2$, and let $\{Q_n\}_{n=d}^{\infty}=\{Q(k,d,n)\}_{n=d}^{\infty}$.
Then
\[
\lim_{n \to \infty} \frac{\gc_{k+1}(Q_n)}{2^{n} \mchoose{n-d}{k}} = 1,
\qquad\text{and}\quad
\lim_{n \to \infty} \frac{\gc_i(Q_n)}{2^{n} \mchoose{n-d}{k}} = 0
\quad (\forall i \ne k).
\]
\end{corollary}

Corollary~\ref{t:asymptotics} shows that the ray spanned by $e_k$
$(2 \leq k \leq \lfloor d/2 \rfloor)$ in $\R^{\lfloor d/2 \rfloor}$ belongs to the closed cone $\mathcal{C}_d$.
Note that this was already known for the ray spanned by $e_{\lfloor d/2 \rfloor}$ because of the existence of neighborly cubical $d$-polytopes, such as $Q(k,2k+2,n)$).
The ray spanned by $e_1$ belongs to this cone because of the existence of stacked cubical $d$-polytopes.
Thus $\mathcal{C}_d$ contains $\mathcal{A}_d$, and \autoref{thm:cones} is proved.

\section{No obvious cubical GLBT}\label{sec:stack}

In~\cite{BabsBC97}, after introducing the definitions of cubical stackedness and cubical neighborliness, the authors show that cubical $1$-stacked $d$-polytopes with at least $n$ vertices exist, for any $n\geq 2^d$ (see~\cite[Corollary 5.6]{BabsBC97}). It is also shown (see~\cite[proof of Proposition 5.5]{BabsBC97}) that if $Q$ is a cubical $k$-stacked $d$-polytope, then $\gc_i(Q)=0$ for $k<i\leq\dhalf{d}$. The converse claim, namely, that $\gc_{k+1}(Q)=0$ implies that $Q$ is cubical $k$-stacked, is false, as our analysis of $Q(k,d,n)$ below shows. This is in apparent contrast with the simplicial GLBT.

\begin{theorem}\label{prop:notStacked}
For any $k\ge 1$ and $n\ge d\ge 2k+4$, the polytope $Q=Q(k,d,n)$ is cubical $k$-neighborly but not cubical $(k+1)$-stacked,
although $\gc_{k+2}(Q)=0$.
In fact, it is not $j$-stacked for any $k+1 \le j \le \dhalf{d}-1$.
\end{theorem}
\begin{proof}
The polytope $Q$ is cubical $k$-neighborly by~\cite[Theorem 16]{JoswZ00} or~\cite[Theorem 3.2]{SanyZ10}; this also follows from our explicit computation of its (short) cubical $g$-vector.
By Theorem~\ref{t:gc_Q}, $\gc_{k+2}(Q)=0$.
Assume by contradiction that $Q$ is cubical $(k+1)$-stacked, so $Q$ has some cubulation $Q'$, namely a subdivision into (combinatorial) cubes, without interior $(d-k-2)$-faces.
Let $C_n$ be the deformed $n$-cube that $Q$ is a projection of.
\begin{lemma}\label{lem:cubulationBySubcubes}
All faces of $Q'$ must be faces of $C_n$.
\end{lemma}
\begin{proof}
Consider an $m$-cube $F$ of $Q'$. Since $Q'$ has no interior $1$-faces, the $1$-skeleton of $F$ is in $Q$, and so also in $C_n$. Now, any $1$-dimensional subcomplex of $C_n$ which is isomorphic to the graph of an $m$-cube, is the $1$-skeleton of an $m$-face of $C_n$.
\end{proof}

Next, the cubulation $Q'$ induces a set of \emph{compatible} triangulations of the vertex figures. We extend the notation for the boundary complex of the vertex figure in a polytope, and denote by $\lk_v(B)$ the simplicial complex whose face lattice is the ideal above the vertex $v$ in the face lattice of $B$, with $B$ a simplicial complex or a cubical complex.

Each $\lk_v(Q')$ is a triangulation of $\vfig{Q}{v}$ with no interior $(d-k-3)$-faces.
Thus the vertex figure of $Q$ at $v$ --- isomorphic to some diamond $D_a$ --- is $(k+1)$-stacked, and by the simplicial GLBT (an easy part of it, see \cite[Thm.2.3(ii)]{MuraN13} or \cite[Thm.2.12]{Bagchi-Datta}), $\lk_v(Q')$ is the unique triangulation obtained from the diamond $D_a$  by inserting all
$(d-1)$-simplices whose $(d-k-3)$-skeleton is contained in the boundary of $D_a$.
(We abuse notation and identify $\partial D_a$ with $\lk_v(Q)$ by the isomorphism given in Lemma~\ref{lem:SZ}.)
This description allows us to determine, for each vertex $v$ of $Q$, the list of $d$-cubes in $Q'$ that contain $v$. The compatibility condition mentioned above is the requirement that if $u$ is a vertex in a $d$-cube from the list of $v$, then the list of $u$ must contain this $d$-cube too.
We show, in steps, that the $(k+1)$-stacked triangulations of the diamonds $D_a$ are incompatible, thus such a cubulation $Q'$ cannot exist.

In order to identify the facets in the unique $(k+1)$-stacked triangulation $S(D_a)$ of $D_a$, we first identify the \emph{missing faces} of $D_a$, namely its minimal non-faces.
Kalai~\cite{Kalai:missing} and Nagel~\cite{Nagel:missing} showed that if $g_i(R)=0$ for a simplicial $j$-polytope $R$, then $R$ has no missing $\ell$-faces for $i\le \ell\le j-i$.
Combined with the simplicial GLBT, a subset $A$ of $d$ vertices of $D_a$ is a facet in $S(D_a)$ iff all its subsets $B$ of size $\leq k+2$ are faces of $D_a$. Since $D_a$ is $k$-neighborly, it is enough to determine the missing faces of $D_a$ of sizes $k+1$ and $k+2$.
To analyze these missing faces of $D_a$ we make an essential use of the following observation:
let $P=\mw(2k,d-2,m)$, and recall Definition~\ref{def:MW}, where $P$ is constructed from the cyclic polytope $C=C(2k,m-d+2+2k)$ and the $(d-2-2k)$-simplex $T$.
\begin{lemma}\label{lem:MWcommuteLEX}
The operations $\mw$ and $\lex$ commute:
	\begin{equation}\label{eq:MWcommuteLEX}
	\lex_a(P)=\langle T\rangle \ast \lk_x(\lex_a(C))\bigcup_{\partial T\ast \lk_x(\lex_a(C))}\partial T\ast \anst_x(\lex_a(C)).
	\end{equation}
\end{lemma}
\begin{proof}
	 By the construction of $P$, for a subset $F$ of $\vertices(P)$ (so that $x\notin F$), we have:
	\begin{itemize}
		\item $F$ is a face of $P$, with $\vertices(T) \subseteq \vertices(F)$, iff $(\vertices(F) \setminus \vertices(T) ) \cup \{x\}$ is the set of vertices of a face $\bar{F}$ of $C$, and
		\item $F$ is a face of $P$, with $\vertices(T) \not\subseteq \vertices(F)$, iff $\vertices(F) \setminus \vertices(T)$  is the set of vertices of a face $\bar{F}$ of $C$.
	\end{itemize}
	
	Now, let $F$ be a face added when pulling $v_1$ in $P$. Then $F=v_1\ast F'$, with $F'$ a face of $P$, and $F'$ a face of $P_1=\conv(v_2,...,v_m)$.
With the notation
above, $\bar{F'}$ is a face of both $C$ and $C_1$, and thus $\bar{F}=v_1\ast \bar{F}'$ is added when pulling $v_1$ in $C$.
	Similarly, when pushing $v_1$ in $P$, we add faces $F$ (and $v_1*F$) of $P_1$ which are non-faces of $P$, and so $\bar{F}$ is a face of $C_1$, and a non-face of $C$, and thus is added when pushing $v_1$ in $C$.
	Reversing the argument shows that faces added when pushing/pulling $v_1$ in $C$ give corresponding added faces in pushing/pulling in $P$.
	When constructing $\lex_a(P)$ and $\lex_a(C)$ sequentially, we always push/pull $v_1$ in a suitable $\mw$-polytope and the corresponding suitable cyclic polytope $C$, hence \eqref{eq:MWcommuteLEX} follows.
\end{proof}

We need some terminology, to be used in the following two propositions:
let $G$ be a subset of the set $\{v_1, \ldots, v_{n-d+2k+1}\}$ of vertices of the cyclic polytope $C$. A {\bf block} $B \subseteq G$ is a maximal (with respect to inclusion) subset of $G$ consisting of consecutive vertices, $B$ is {\bf inner} if $\min B>v_1$ and $\max B<v_{n-d+2k+1}=x$, and $B$ is {\bf even} / {\bf odd} if its size is. We call a subset of vertices of $C$ {\bf isolated} if no two of them are consecutive. Recall that $p$ denotes the apex of the diamond $D_a$ of $P=MW(2k,d-2,n-1)$.

\begin{proposition}\label{prop:missing_faces}
For any $1\le a\le n-d+1$, the missing faces in $D_a$ of size $k+1$ or $k+2$ are exactly the sets of the following two types:
\begin{enumerate}[(1)]
	\item $F=\{p\} \cup F'$, with $F'$ of size $k+1$, consisting of isolated vertices of $C$, $x\notin F'$, and $\min F'=a$; or
	\item $F$ is of size $k+1$, consisting of isolated vertices of $C$, $x\notin F$, $\min F\neq a$.
\end{enumerate}
\end{proposition}
\begin{proof}
Any missing face $F$ of $D_a$ must contain a missing face of $P$, and
is either of the form

(i) $\{p\} \cup F'$ where $F'$ is an interior face of the ball $\lex_a(P)$, or

(ii) a missing face of $\partial P$ which is not a face of $\lex_a(P)$, or 

(iii) a missing face of $\lex_a(P)$ containing a missing face $F'$ of $\partial P$ as a strict subset.

Using Lemma~\ref{lem:MWcommuteLEX} we can replace $P$ by $C$ in (i), (ii) and (iii), and $F$ contains a missing face of $C$ not containing $x$.

By the Gale evenness criterion, a subset of $i$ vertices of $C$ forms a face of $C$ iff it contains at most $2k-i$ inner odd blocks. It follows that the missing faces of $C$ are exactly the sets of $k+1$ isolated vertices, containing at most one element of $\{v_1,x\}$.

Let $F'$ be a missing face of $C$, $x\notin F'$, with $\min F'=v_i$, and denote $C_i=\conv\set{v_j}{j\geq i}$, a cyclic polytope contained in $C$.

If $i=a$ then $F'$ is a face in $\lex_a(C_a)$, thus $F'$ is also in $\lex_a(C)$, yielding a missing face $F=F' \cup \{v_a\}$ of $D_a$ of type (i).

If $i>a$, then $F'$ is a missing face in $\lex_a(C)$, hence $F=F'$ is a missing face of $D_a$ of type (ii).

If $i<a$, $C_{i+1}$ blocks $v_i$ from seeing $F'\setminus v_i$, so $F'$ is missing in $\lex_a(C)$, and so $F=F'$ is a missing face of $D_a$ of type (ii).

Now, we will show that missing faces of type (iii) do not exist, and thus the characterization in the statement will follow.
Indeed, assume $F$ to be of type (iii), and denote $\{z\}= F\setminus F'$ (indeed this difference set must be a singleton as $|F'|=k+1<|F|\le k+2$ and $F'\subseteq F$).
As $F'\in\lex_a(C)$, the discussion above shows $\min(F')=v_a$.
If $z\neq \min F$, then $\min F=v_a$, and $z$ is adjacent to some vertex of $F'$ (else $F\setminus\{v_a\}$ would be missing in $\lex_a(C)$).
But there is no missing face containing $v_a$ in $\lex_a(C_a)$
so $F$ is in $\lex_a(C_{a})$, hence also in $\lex_a(C)$, a contradiction.
If $z=\min F$, then $F\setminus \{v_a\}$ is a missing face of $D_a$ of type (ii), a contradiction.

To conclude, note that the items (i) and (ii), with $C$ replacing $P$, are restatements of (1) and (2), respectively.
\end{proof}

Proposition~\ref{prop:missing_faces} and the discussion preceding Lemma~\ref{lem:MWcommuteLEX} imply the following description of the $(d-1)$-faces of the $(k+1)$-stacked triangulation $S(D_a)$.

\begin{proposition}\label{prop:d-faces}
The sets of $d$ vertices that form the $(d-1)$-simplices of the $(k+1)$-stacked triangulation $S(D_a)$ of $D_a$ are exactly the sets of the following two types:
\begin{enumerate}[(I)]
	\item $\{p\}\cup \left\{2k \text{-set in } C\setminus \{x\}\text{ consisting of even blocks}\right\}\cup\, \vertices(T)$, or
	\item $\{a\}\cup \left\{2k \text{-set $F$ in } C\setminus \{x\}\text{ consisting of even blocks}, \min F>a\right\}\cup\, \vertices(T)$.
\end{enumerate}
\end{proposition}
\begin{proof}
Let $Y$ be a $d$-set that forms a facet of $S(D_a)$. Then $Y$ cannot contain more than $2k+1$ vertices of $C\setminus \{x\}$, because then it would contain two disjoint missing faces of $C$, one of which must be of type (2). Clearly, $Y$ has to contain at least $2k$ vertices of $C\setminus \{x\}$. So, by Proposition~\ref{prop:missing_faces}, $Y$ is either of the form
\begin{itemize}
	\item $2k$ vertices of $C\setminus \{x\}$, the apex $p$, and all of the vertices of $T$, or
	\item $2k+1$ vertices of $C\setminus \{x\}$, and all of the vertices of $T$.
\end{itemize}

If $Y$ has $2k+1$ vertices of $C\setminus \{x\}$, then it contains $k+1$ isolated vertices of $C\setminus \{x\}$, so $\min Y=a$, since $Y$ contains no missing faces of type (2), and $p\notin Y$ since $Y$ contains no missing faces of type (1). These are the $d$-sets of type (II) in the statement. Otherwise, $Y$ has $2k$ vertices of $C\setminus \{x\}$, and hence must contain $p$ and all vertices of $T$, and since $Y$ does not contain missing faces of type (1) or (2), it must be of type (I).
\end{proof}

We are now ready to complete the proof of Theorem~\ref{prop:notStacked}.
Consider a vertex $v_{\sigma}$ of $Q$ with $\sigma\in\{+,-\}^n$, and with $a=\min\{i\mid \sigma_i=+\}< n-d+1$.
Lemma~\ref{lem:SZ} says that $\vfig{Q}{v_{\sigma}}\cong D_a$, with the vertex $v_i\in D_a$ corresponding to the vertex in $\vfig{Q}{v_{\sigma}}$ obtained by flipping the $(i+1)$-th coordinate of $\sigma$. Consider $v_{\sigma'}$, where $\sigma'$ is obtained by flipping the $a$-th coordinate of $\sigma$. Then $v_{\sigma'}$ is a neighbor of $v_{\sigma}$ with $\vfig{Q}{v_{\sigma'}}\cong D_b$, for some $b>a$.

By Proposition~\ref{prop:d-faces}, there is a $(d-1)$-face $F$ in $S(D_a)$ of type (II) (e.g., consisting of $a$, $\vertices(T)$, and a single $2k$-block in $C \setminus \{x\}$). It corresponds to a $d$-cube $D$ in $Q'$. The same $d$-set of coordinates $F$ corresponds to the face of $S(D_b)\cong\lk_{v_{\sigma'}}(Q')$ which corresponds to $D\in Q'$. But $F$ is of neither type (I) nor type (II) in $D_b$, and so the triangulations are incompatible.

This completes the proof that $Q$ is not $(k+1)$-stacked.
For $k+1<j\le \dhalf{d}-1$, if $Q$ were $j$-stacked, then the induced $j$-stacked triangulation of any diamond $D_a$ would again be the same $S(D_a)$ given by the GLBT, and thus the analysis above yields a contradiction as before.
\end{proof}

\section{Concluding remarks}\label{sec:concluding_remarks}

\subsection{Cubical $k$-stacked $d$-polytopes}
As mentioned in Theorem~\ref{prop:notStacked}, $Q(k,d,n)$ is not cubical $j$-stacked for any $k+1\leq j\leq \dhalf{d}-1$. It is then natural to ask:
do cubical $k$-stacked, and not $(k-1)$-stacked, $d$-polytopes actually exist?

The {\em $k$-elementary} cubical $d$-polytopes $C^d_k$ constructed for any $0\leq k\leq d$ by Blind and Blind in~\cite{BlinB95} provide a positive answer to this question. 
For $1 \le k \le \lfloor d/2 \rfloor$, $C^d_k$ is $k$-stacked by construction and its $\gc$-vector is
\[
g^c_i(C^d_k)
= \sum_{j=1}^{k} 2^{d-j} \binom{j-1}{i-1}
\qquad (1 \le i \le \lfloor d/2 \rfloor)
\]
so that, in particular, $g^c_k(C^d_k) = 2^{d-k} > 0$ and hence $C^d_k$ is not $(k-1)$-stacked.

The following question, implicit in~\cite{BabsBC97}, asks for a sequence of cubical $k$-stacked $d$-polytopes with $\gc$-vector approaching the ray spanned by $e_k$. It is still unanswered.
\begin{question}\label{Q:existBBC?}
Let $2\leq k\leq\dhalf{d}-1$.
Does there exist a sequence of cubical $k$-stacked $d$-polytopes with dominant $k$-th coordinate $g^c_k$ of the cubical $g$-vector?
\end{question}

\subsection{The Cubical Lower Bound Conjecture}
Jockusch studied the lower and upper bound problems for cubical polytopes in~\cite{Jock93}, where he stated a {\em Cubical Lower Bound Conjecture}, rephrased as follows:
\begin{conjecture}[CLBC, \cite{Jock93}]\label{conj:JCLBC}
	Let $Q$ be a cubical $d$-polytope with $n$ vertices. Then
	\begin{align*}
	f_k(Q) &\ge \left[ \binom{d}{k} + \binom{d-1}{k-1} \right] 2^{-k} n - \binom{d-1}{k-1} 2^{d-k} \quad (1 \le k \le d-2),\\
	f_{d-1}(Q) &\ge (2d-2) 2^{-(d-1)} n - (2d-4).
	\end{align*}
\end{conjecture}

Jockusch remarks that the CLBC for edges is as strong as the general statement. This reduction is a cubical analog of the {\em MPW-reduction} (see \cite[Theorem~1]{Barn73}), and it follows that Conjecture~\ref{conj:JCLBC} is equivalent to the $k=1$ case, which in turn can be rewritten as
\begin{conjecture}[CLBC]\label{conj:CLBC}
	If $Q$ is a cubical $d$-polytope with $n$ vertices then
	\begin{equation*}\gc_2(Q)\geq 0.\end{equation*}
\end{conjecture}

\bibliographystyle{myalpha}
\bibliography{cubical}

\begin{thebibliography}{McM71}

\bibitem[Adi96]{Adin96}
R.~M. Adin.
\newblock A new cubical {$h$}-vector.
\newblock {\em Discrete Math.}, 157(1-3):3--14, 1996.

\bibitem[Bar73]{Barn73}
D.~Barnette.
\newblock A proof of the lower bound conjecture for convex polytopes.
\newblock {\em Pacific J. Math.}, 46:349--354, 1973.

\bibitem[BB95]{BlinB95}
G.~Blind and R.~Blind.
\newblock The cubical {$d$}-polytopes with fewer than {$2^{d+1}$} vertices.
\newblock {\em Discrete Comput. Geom.}, 13(3-4):321--345, 1995.

\bibitem[BBC97]{BabsBC97}
E.~K. Babson, L.~J. Billera, and C.~S. Chan.
\newblock Neighborly cubical spheres and a cubical lower bound conjecture.
\newblock {\em Israel J. Math.}, 102:297--315, 1997.

\bibitem[BD13]{Bagchi-Datta}
B.~Bagchi and B.~Datta.
\newblock On {$k$}-stellated and {$k$}-stacked spheres.
\newblock {\em Discrete Math.}, 313(20):2318--2329, 2013.

\bibitem[BL81]{BillL81}
L.~J. Billera and C.~W. Lee.
\newblock A proof of the sufficiency of {M}c{M}ullen's conditions for
  {$f$}-vectors of simplicial convex polytopes.
\newblock {\em J. Combin. Theory Ser. A}, 31(3):237--255, 1981.

\bibitem[Joc93]{Jock93}
W.~Jockusch.
\newblock The lower and upper bound problems for cubical polytopes.
\newblock {\em Discrete Comput. Geom.}, 9(2):159--163, 1993.

\bibitem[JZ00]{JoswZ00}
M.~Joswig and G.~M. Ziegler.
\newblock Neighborly cubical polytopes.
\newblock {\em Discrete Comput. Geom.}, 24(2-3):325--344, 2000.
\newblock The Branko Gr\"unbaum birthday issue.

\bibitem[Kal94]{Kalai:missing}
G.~Kalai.
\newblock Some aspects of the combinatorial theory of convex polytopes.
\newblock In {\em Polytopes: abstract, convex and computational ({S}carborough,
  {ON}, 1993)}, volume 440 of {\em NATO Adv. Sci. Inst. Ser. C Math. Phys.
  Sci.}, pages 205--229. Kluwer Acad. Publ., Dordrecht, 1994.

\bibitem[Kle11]{Klee11}
S.~Klee.
\newblock Lower bounds for cubical pseudomanifolds.
\newblock {\em Discrete Comput. Geom.}, 46(2):212--222, 2011.

\bibitem[McM71]{McMu71}
P.~McMullen.
\newblock The numbers of faces of simplicial polytopes.
\newblock {\em Israel J. Math.}, 9:559--570, 1971.

\bibitem[MN13]{MuraN13}
S.~Murai and E.~Nevo.
\newblock On the generalized lower bound conjecture for polytopes and spheres.
\newblock {\em Acta Math.}, 210(1):185--202, 2013.

\bibitem[MW71]{McMuW71}
P.~McMullen and D.~W. Walkup.
\newblock A generalized lower-bound conjecture for simplicial polytopes.
\newblock {\em Mathematika}, 18:264--273, 1971.

\bibitem[Nag08]{Nagel:missing}
U.~Nagel.
\newblock Empty simplices of polytopes and graded {B}etti numbers.
\newblock {\em Discrete Comput. Geom.}, 39(1-3):389--410, 2008.

\bibitem[Nev07]{nevo_VKobs}
E.~Nevo.
\newblock Higher minors and {V}an {K}ampen's obstruction.
\newblock {\em Math. Scand.}, 101(2):161--176, 2007.

\bibitem[San05]{Sany05}
R.~Sanyal.
\newblock On the combinatorics of projected deformed products.
\newblock Diploma Thesis, available online at
  \url{http://page.mi.fu-berlin.de/sanyal/DiplomaThesis.pdf}, 2005.

\bibitem[Sta80]{Stan80}
R.~P. Stanley.
\newblock The number of faces of a simplicial convex polytope.
\newblock {\em Adv. in Math.}, 35(3):236--238, 1980.

\bibitem[SZ10]{SanyZ10}
R.~Sanyal and G.~M. Ziegler.
\newblock Construction and analysis of projected deformed products.
\newblock {\em Discrete Comput. Geom.}, 43(2):412--435, 2010.

\bibitem[Zie95]{Zieg95}
G.~M. Ziegler.
\newblock {\em Lectures on polytopes}, volume 152 of {\em Graduate Texts in
  Mathematics}.
\newblock Springer-Verlag, New York, 1995.

\end{thebibliography}

\end{document}